\documentclass[10pt]{amsart}
\usepackage{footnote}
\makesavenoteenv{tabular}
\usepackage{amssymb,multicol}
\usepackage{mathtools}
\usepackage[pdfpagelabels=true,plainpages=false]{hyperref}
\usepackage[paperheight=11in,paperwidth=8.5in,
        top=1.0in,left=1.0in,right=1.00in,textheight=9.00in,]{geometry}
\usepackage{fancyhdr}
\usepackage{paralist}
\usepackage{cite}
\usepackage{pictex}

\begin{document}

\theoremstyle{plain}
\newtheorem{thm}{Theorem}[section]
\newtheorem{prop}[thm]{Proposition}
\newtheorem{lemma}[thm]{Lemma}
\newtheorem{prin}[thm]{Principle}
\newenvironment{thmLabeled}[1]
{\begin{thm}
\label{#1}{\thethm}
\label{page: #1}
}
{\end{thm}}

\newenvironment{propLabeled}[1]
{\begin{prop}
\label{#1}{\theprop}
\label{page: #1}
}
{\end{prop}}

\newenvironment{prinLabeled}[1]
{\begin{prin}
\label{#1}{\theprin}
\label{page: #1}
}
{\end{prin}}

\newenvironment{lemmaLabeled}[1]
{\begin{lemma}
\label{#1}{\thelemma}
\label{page: #1}
}
{\end{lemma}}

\theoremstyle{definition}
\newtheorem{defn}[thm]{Definition}
\newcommand{\definedterm}[1]{\textit{#1}}
\newenvironment{defLabeled}[1]
{\begin{defn}
\label{#1}{\thedefn}
\label{page: #1}
}
{\end{defn}}
\newtheorem{conv}[thm]{Convention}
\newtheorem{decl}[thm]{Declaration}

\theoremstyle{remark}
\newtheorem{obs}[thm]{Observation}
\newtheorem{notation}[thm]{Notation}
\newtheorem{rem}[thm]{Remark}
\newtheorem{exercise}[thm]{Exercise}
\newtheorem{example}[thm]{Example}
\newtheorem{prologue}[thm]{Prologue}
\newtheorem{discussion}[thm]{Discussion}
\begin{Large}

\newcommand\ms{\par\medskip}
\newcommand\bs{\par\bigskip}
\newcommand\nl{\newline}
\newcommand\np{\newpage}
\newcommand\cl{\centerline}
\newcommand\myframe[1]{ \fbox {\vtop{\vskip  6pt\hbox{\hskip  6pt {#1} \hskip  6pt}} } }
\newcommand\titles{\large\bf}
\newcommand\proves{\vdash }
\newcommand\bicond{\leftrightarrow }
\newcommand\cond{\rightarrow }

\begin{center}
\textbf{Andrew Wiles' Proof of Fermat's Last Theorem, As Expected, Does Not Require a Large Cardinal Axiom}

\bigskip
A Discussion of Collin McLarty's \textit{The Large Structures of Grothendieck
Founded on Finite-Order Arithmetic}
\cite{McLarty2020}
\footnote{Presentation to the Indiana University, Bloomington, Logic Seminar}

\bigskip
William Wheeler

\end{center}

\bigskip
\bigskip
Andrew Wiles' proof of Fermat's Last Theorem\cite{Wiles1995},
with an assist from Richard Taylor \cite{TaylorWiles1995}, focused renewed 
attention on the foundational question of whether the use of Grothendieck's 
\textit{Universes} in number theory
entails that the results proved therewith make essential use 
of the large cardinal axiom that there
is a  strongly inaccessible cardinal greater than $\aleph_0$, or more generally,
that every cardinal is less than some strongly inaccessible cardinal.
(The latter is equivalent over Zermelo-Fraenkel set theory with the Axiom of Choice (ZFC)
to Grothendieck's axiom U that every set is contained in a Grothendieck universe.)

Every number theorist, including Grothendieck himself, has believed that number theoretic
results proved using Grothendieck universes could be proved without using them if one were
willing to make the effort.  But, in print, few do.

If one traces back through the references in Wiles proof,
one finds that the proof does
depend upon explicit use of Grothendieck's universes in \cite{GrothendieckDieudonne1961} (see
\cite{McLarty2010}, page 362 (middle)).  Thus, \textit{prima facie}, it appears that the proof
of Fermat's Last Theorem depends upon a foundation that is strictly stronger than ZFC.

Colin McLarty in \cite{McLarty2020} removes this appearance by demonstrating that all of 
Grothendieck's ``large'' tools, i.e., entities whose construction depended upon Grothendieck's
universes, can instead be founded on a ``fragment of ZFC with the logical strength
of Finite-Order Arithmetic.

The goal of this presentation is to present overviews both of the history
of Fermat's Last Theorem and of McLarty's foundation for Grothendieck's large tools.

\bs
\section {\textbf{Milestones in the Proof of Fermat's Last Theorem}}

\bigskip
\begin{small}
\begin{tabular}{r|l|l|l|l|}
&&Algebraic&Analytic&\\
&Number&Number&Number&Algebraic\\
Date&Theory&Theory&Theory&Geometry\\
\hline
1630&Fermat's Conjecture&&&\\
&Proof for $n=4$&&&\\
\hline
1730&&&Euler studies&\\
&&&Real $\zeta(s)$&\\
\hline
1760&Euler proves&&&\\
&FLT for $p=3$&&&\\
\hline
1796&&&&Gauss using \\
&&&&Gaussian sums determines \\
&&&&number of solutions of\\
&&&&$ax^3-by^3 \equiv 1 mod p$\\
\hline
1801&&&&Gauss determines \\
&&&&number of solutions of\\
&&&&$ax^4-by^4 \equiv 1 mod p$\\
\hline
\end{tabular}
\newpage

\begin{tabular}{r|l|l|l|l|}
&&Algebraic&Analytic&\\
&Number&Number&Number&Algebraic\\
Date&Theory&Theory&Theory&Geometry\\
\hline
1825&Legendre and&&&\\
&Dirichlet independently&&&\\
&prove FLT for $p=5$&&&\\
\hline
1837&&&Dirichlet introduces&\\
&&&\textit{Dirichlet L-series}&\\
\cline{4-4}
&&&Riemann extends $\zeta{s}$&\\
&&&to a meromorphic function&\\
&&&on complex plane; proves &\\
&&&functional equation;&\\
&&&conjectures all nontrival&\\
&&&zeroes have real part $1/2$&\\
\hline
1839&Lam\'e proves&&&\\
&FLT for $p = 7$&&&\\
\hline
1847&Lam\'e presents &&&\\
&false proof for all $n$&&&\\
\cline{2-3}
&Kummer points out&Kummer initiates&\\
&Lam\'e's error&Algebraic Number&&\\
&&Theory&&\\
\hline
1930&Vandiver proves&&&\\
&FLT for $p < 157$&&&\\
\hline
1949&&&&Weil revived Gauss'\\
&&&&work and determines\\
&&&&number of solutions of\\
&&&&$a_0x_0^{n_0} + a_1x_1^{n_1} + \cdots a_rx_r^{n_r}=0$\\
&&&&in a finite field.\\
&&&&formulates the Weil Conjectures\\
\hline
1954&Vandiver proves&&&\\
&FLT for $p<2621$&&&\\
\hline
1955&&&&Taniyama proposes\\
&&&&initial version of the\\
&&&&Shimura-Taniyama-Weil\\
&&&&conjecture\\
\hline
1958&&&&Grothendieck announces\\
&&&&his cohomology results\\
&&&&at Int.~Cong.~of~Math.\\
\hline
1965&&&&Grothendieck proves\\
&&&&first and second\\
&&&&Weil conjectures\\
\hline
1960's&&&&Shimura and Weil refine\\
&&&&the Shimura-Taniyama-Weil\\
&&&&conjecture\\
\hline
1979&&&&Deligne proves third\\
&&&& Weil Conjecture\\
\hline
1982&&&&Frey conjectures\\
&&&&a solution of Fermat's\\
&&&&equation would yield\\
&&&&counterexample to\\
&&&&Shimura-Taniyama-Weil\\
&&&&conjecture\\
\hline
1986&&&&Ribet proves Frey's\\
&&&&conjecture, by then known\\
&&&&as the epsilon conjecture\\
\hline
1993&FLT proved for&&&\\
&$p<4{,}000{,}000$&&&\\
\hline
1994&&&&Wiles presents proof\\
&&&&of FLT at Cambridge\\
\hline
1995&&&&Wiles and Taylor\\
&&&&publish full proof of\\
&&&&FLT\\
\hline
\end{tabular}
\end{small}

\subsection {History 1630--1847}

In the 1630's, Fermat,in the margin of his copy of Diophantus' \textit{Arithmetica} (250 AD)
beside Problem 8 of Book II, (which asked to write a number that is a square
as the sum of two squares), wrote
\begin{quote}
It is impossible to separate a cube into two cubes or a fourth power into two fourth 
powers or, in general, any power greater than the second into powers of like degree. 
I have discovered a truly marvelous demonstration, which this margin is too narrow to contain.
(\cite{RibetHayes1994}, pp. 145--146)
\end{quote}

Fermat did publish a proof of his theorem for fourth powers (i.e, $n=4$) using infinite
descent.  

Note that if Fermat's Last Theorem holds for a positive integer $k$, then it holds
for all multiples $mk$
of $k$, because 
$$x^{mk} + y^{mk} = z^{mk} \quad\text{entails}\quad (x^m)^k + (y^m)^k = (z^m)^k\,.$$
Accordingly, because every integer greater than 2 is divisible by 4 or an odd prime,
and Fermat had proved his theorem for $n=4$, it remained ``only'' to prove the theorem for
odd primes.

Circa 1760, Euler published a proof for the case $p=3$, although there was a gap that
went unnoticed at the time.

Circa 1825, Legendre and Dirichlet independently proved the case $p=5$.

In 1839, Lam\'e published a proof for $p=7$.

So, prior to 1840, Fermat's Last Theorem had been proved only for $n=3$, $4$, $5$,and $7$
and their multiples.

In 1847, Lam\'e presented a purported proof for all $n$ to the March 1, 1847, meeting of the
Paris Academy of Sciences.  BUT it relied upon a false assumption that the unique factorization 
principle holds in the cyclotomic integers generated by the $p^{th}$ roots of unity (i.e., solutions
of $x^p = 1$) for every prime $p$ (\cite{Kleiner2000}, page 24).

Several months later, this error was pointed out by Kummer, who three years earlier in his 
dissertation had shown that unique factorization failed in general in rings of cyclotomic 
integers.  But by introducing \textit{ideal complex numbers}, unique factorization could 
be recovered.  With these, Kummer had been able to prove Fermat's Last Theorem for all
$n< 100$ (\cite{Kleiner2000}, page 24).

\subsection {Algebraic Number Theory, 1847--1994}

This led to the development of Algebraic Number Theory by Kummer and Dedekind.

This eventually led to proofs for $p<157$ by Vandiver by 1930 and then for $p<2521$ again
by Vandiver in 1954 with the assistance of an early computer.  The value of $p$ was
pushed to $p<4,000,000$ by 1993 {\cite{Kleiner2000}, page 27}.

\bigskip
\subsection{Analytic Number Theory} 

Leonhard Euler, in his thesis \cite{Euler1737}, investigate the real zeta function (i.e., 
a function from the real numbers to the real numbers)
defined by the
power series 
$$\zeta(s) = 1 + \frac{1}{2^s} + \frac{1}{3^s} + \frac{1}{4^s} + \frac{1}{5^s} + \cdots $$
and showed that 
\begin{align*}
\zeta(s) &=\frac{2^s\cdot 3^s \cdot 5^s \cdot 7^s\cdot 11^s \cdots}
{(2^s -1)(3^s-1)(5^s-1)(7^s-1) (11^s-1)\cdots} \\
&= \frac{1}{1-(1/2^s} \times \frac{1}{1-(1/3^s)} \times \frac{1}{1-(1/5^s)}
\times \frac{1}{1-(1/7^s)} \cdots \, .
\end{align*}
More concisely,
$$\zeta(s)= \sum_{n=1}^\infty \frac{1}{n^s} = \prod_{p~\text{prime}} \frac{1}{1-(1/p^s)} \,.$$
Several years earlier, in 1734, he had computed that  $\zeta(2) = \frac{\pi^2}{6}$.
Euler's work is considered by some to have been the start of analytic number theory.

In 1837, Peter Gustav Lejeune Dirichlet generalize Euler's $\zeta$ function to 
\textit{Dirichlet $L$-series}
$$L(s, \chi) = \sum_{n=1}^{\infty} \frac{\chi(n)}{n^s}$$ where $s$ is a complex number with real part > 1 and $\chi$ is a function from the integers into the complex numbers such that
\begin{enumerate}
\item $\chi(ab)=\chi(a)\chi(b)$ for all integers $a$ and $b$ and
\item for some positive integer $m$ (called the modulus of $\chi$) and for all integers $a$,
\begin{enumerate}
\item $\chi(a) =
\begin{cases} = 0&\text{if }  $gcd(a,m) > 1$\\
\ne 0&\text{if } $gcd(a,m)=1$\\
\end{cases}$, 
where $gcd$ denotes the greatest common divisor, and
\item $\chi(a + m) = \chi(a)$ .
\end{enumerate}
\end{enumerate}
Dirichlet introduced these these $L$-series for his proof that  any arithmetic progression
$\{a + bn : n \ge 0\}$ for which $a$ and $b$ are relatively prime contains infinitely many prime
numbers(\cite{Dirichlet1837}).  As for Euler's $\zeta$ function, which is a Derichlet $L$-series with $\chi(a)=1$
for all $a$ and with $s$ restricted to real numbers,
$$L(s,\chi) = \sum_{n=1}^\infty \frac{\chi(n)}{n^s}
=\prod_{p}\frac{1}{1-\chi(p)p^{-s}}$$ for complex $s$ with real part $> 1$, where the product
is over all primes. Derichlet also showed that the $L$-functions satisfy a functional 
equation.  Some consider this to have been the start of analytic number theory.

Also, in 1837, Bernhard Riemann (\cite{Riemann1837}) showed that the zeta function defined by
$$\zeta(s) = \sum_{n=1}^\infty \frac{1}{n^s} = \frac{1}{1^s} + \frac{1}{2^s} + \frac{1}{3^s} + \cdots$$ for all complex numbers $s$ with real parts > 1 could be continued analytically to 
a meromorphic function on the complex plane with a simple pole at $s=1$.  As before,
$$\zeta(s) = \sum_{n=1}^\infty \frac {1}{n^s} = \prod_{p\ \text{prime}} \frac{1}{1-p^{-s}}$$ for
$s$ with real part $> 1$.  Riemann showed that the zeta function satisfied a functional
equation
$$ \zeta(s) = 2^s \pi^{s-1} \sin\left(\frac{\pi s}{2}\right) \Gamma(1-s) \zeta(1-s)$$
where $\Gamma(s)$ is the gamma function.  Because of the $\sin$ factor, $\zeta(s) = 0$
for every negative even integer; these zeroes are know as the \definedterm{trivial zeroes}.
Riemann hypothesized that all other zeroes of the zeta function have a real part = 1/2.
This hypothesis, known today as the \definedterm{Riemann Hypothesis}, is still unproven
and is considered to be the most important unsolved problem in mathematics.

\bigskip
\subsection {Algebraic Geometry and the Weil Conjectures}
In 1796, on October 1, Carl Friedrich Gauss published a paper (\cite{Gauss358})
in which he introduced Gaussian sums
and determines the Gaussian sums of order 3 for a prime of the form $p=3n+1$
and also determined the numbers of solutions of all congruences of the form 
$ax^3 - by^3 \equiv 1 \mod p$.  He proceeded similarly in his 
first memoir on biquadratic residues (\cite{GaussBQR}) and determined the numbers
of solutions of all congruences of the form $ax^4 - by^4 \equiv 1 \mod p$ for primes
of the form $p=4n+1$.  

\bigskip
In 1949, Andre Weil (\cite{Weil1949}) undertook to renew attention to Gauss' work
by presenting a ``complete exposition'' (Weil's self-characterization)
of the numbers of solutions of equations of the form
$$a_0x_0^{n_0} + a_1x_1^{n_1} + \cdots + a_r x_r^{n_r} = b$$
in a finite field. 

This exposition motivates and concludes with four conjectures concerning the zeroes
of varieties defined over a finite field $k=F_q$ with $q$ elements.
These conjectures are now known as the ``Weil conjectures''.
First, Weil defined an analogue of the Riemann zeta function as follows.  
Let $X$ be a finite set of polynomial
equations with coefficients in $k$.  
Let $\tilde k$ be the algebraic closure of $k$.  
For integer $r \ge 1$, let $N_r$ be the number of solutions of $X$
in the subfield $k_r = F_{q^r}$ of $\tilde k$ (where $k_r = F_{q^r}$ has $q^r$ elements).
Define a zeta function $Z(t)$  by
$$Z(t) = Z(X; t) = \exp\left(\sum_{r=1}^\infty \left(N_r\frac{t^r}{r}\right)\right)$$.
Weil's conjectures concerning this zeta function were:
\begin{enumerate} 
\item $Z(t)$ is a rational function of $t$, i.e., it is a quotient of polynomials 
with rational coefficients.
\item $Z(t)$ satisfies a specific functional equation.
\item $Z(t)$ satisfies an analogue of the Riemann hypothesis.
\item $Z(t)$ determines the correct Betti numbers.
\end{enumerate}

Although described here for varieties, the Weil conjectures pertain more broadly to arbitrary
schemes.

Weil had (already) proved his conjectures for curves (\cite{Weil1948}).
\bigskip

For higher dimensional varieties, the first conjecture concerning the rationality of the zeta
function and the second conjecture on its functional equation were proved by B.~Dwork in 1960
(\cite{Dwork1960}).

\bigskip
A.~Grothendieck, inspired by some ideas of J.~P.~Serre, began his development of \'etale cohomology (with Emil Artin) announced in \cite{Grothendieck1958} and exposited in \cite{GrothendieckDieudonne1961}.
This led to Grothendieck's proof of the rationality of the zeta function
and its functional equation for general schemes in \cite{Grothendieck1965}.

\bigskip
In 1974, P.~Deligne proved the general analogue of the Riemann hypotheses for schemes
(\cite{Deligne1974}).

\bigskip
\subsection {Elliptic Curves and Modular Curves}
The study of elliptic curves and modular curves has a long history in mathematics

An elliptic curve is a "plane curve" that is given by an equation of the form
$$y^2=x^3 +ax^2 + bx+c$$ where $a$, $b$, and $c$ are integers, or rational numbers, 
or real numbers, or even complex numbers, and the cubic polynomial on the right side
has distinct roots

A cubic curve is said to be modular if it has a finite covering by a modular curve of the form $X_0(N)$.
The full details are quite technical.  In essence, modularity entails that there is a formula
for the number of solutions of the cubic equation in each finite field (\cite{Cipra1994}).

For a clear, fuller discussion, the reader is referred 
to the survey article by Henri Darmon (\cite{Darmon1999}).

The Shimura-Taniyama-Weil conjecture, the first version of which was put forth by Taniyama in
1955, and which was
subsequently refined and clarified by Shimura and Weil in the 1960's,
asserted that every elliptic curve is modular.  

In 1982, Gerhard Frey conjectured that if $a^p + b^p = c^p$ for nonzero integers $a$, $b$, and $c$
and a prime $p > 2$, i.e., if $a$, $b$, $c$ and $p$ were a counterexample to Fermat's Last Theorem,
then the elliptic curve $y^2 = x(x-a^p)(x+b^p)$ would not be modular,
i.e., would be a counterexample to the Shimura-Taniyama-Weil conjecture (\cite{Frey1982}).
These curves are now referred to as \textit{Frey Curves}.
Frey's conjecture was refined by J.-P.~Serre, and became known as the \textit{Epsilon Conjecture}.

In 1986, Kenneth Ribet proved the \textit{Epsilon Conjecture} (\cite{Ribet1990}).

\bigskip
\subsection {Andrew Wiles -- Finally a Proof}

\bigskip
The Frey curves possess a technical property referred to as semistability.

In June, 1994, Andrew Wiles presented his proof that any semistable elliptic curve is modular
at a conference in Cambridge.(\cite{Kleiner2000},  pp. 27-35).
Thus, no Frey curve exists, and this proves Fermat's Last Theorem.

There was a gap in Wiles' original proof, that was fixed as a result of discussions between Wiles
and his former student Richard Taylor.

The finished proof of Fermat's Last Theorem is contained in 
\cite{Wiles1995} and \cite{TaylorWiles1995}.

\bigskip
\subsection {Post Wiles -- The Modularity Theorem}

\bigskip
Using ideas from Wiles, Christophe Breuil, Brian Conrad, Fred Diamond, and Richard Taylor,
in a series of papers (\cite{Diamond1996}, \cite{ConradDiamondTaylor1999},
\cite{BreuilConradDiamondTaylor2001}) completed the proof of the 
Shimura-Taniyama-Weil conjecture, i.e., that every elliptic curve over the rationals is
modular.  The conjecture is now known as the \textbf{Modularity Theorem}.

\bigskip\bigskip
\section  {\textbf{McLarty's Foundations for Grothendieck's Large Structures}}

\bigskip
\subsection {Overview}

\bigskip
\strut
From a foundational perspective, what one has from Wiles' proof is that
\[ \text{ZFC + U} \proves \text{Fermat's Last Theorem} \ ,\]
i.e., ZFC + U proves Fermat's Last Theorem,
where ZFC denotes Zermelo-Fraenkel set theory together with the Axiom of Choice
and U denotes Grothendieck's Axiom of Universes.

Although Wiles does not explicitly invoke Grothendieck Universes, McLarty 
(\cite{McLarty2010}, page 362) has shown that Wiles references Mazur (\cite{Mazur1977},
\S II.3) but that therein Mazur does not give complete proofs but rather references
Grothendieck and Dieuddonn\'e (\cite{GrothendieckDieudonne1961}), which is devoted to
Grothendieck Universes.

Grothendieck wanted a Universe to be a set that was ``large enough that the habitual operations of set theory do not go outside" it (\cite{SGA1} VI.1, page 146). Grothendieck (\cite{SGA4}, vol. I, page 196)
gave a proof that, in ZFC, being a Grothendieck universe was the same as being the
set $V_\kappa$ of sets of rank less than $\kappa$ for some uncountable, strongly inaccessible 
cardinal $\kappa$.

A cardinal $\kappa$ is strongly inaccessible if (i) it is not the union of 
$<\kappa$ sets each of cardinality $<\kappa$, i.e., it is its own cofinality,
and (ii) if $x$ is a set of
cardinality $< \kappa$, then its power set also has cardinality $<\kappa$.

But if $\kappa$ is an uncountable, strongly inaccessible cardinal,
then $V_\kappa$ is a model of ZFC.  Thus, by G\"odel's Incompleteness 
Theorem, ZFC cannot prove the existence of
an uncountable, strongly inaccessible cardinal much less such a $V_\kappa$, because then ZFC, 
if it is consistent, would prove its own consistency.

Thus, Grothendieck's Axiom of Universes, 
which asserts that every set is a member of some Grothendieck
Universe is  equivalent to the large cardinal axiom that every cardinal is less than some
uncountable, strongly inaccessible cardinal.

No one has believed that Grothendieck universes 
are essential to the proof of Fermat's Last Theorem;
rather everyone has believed that, with enough work, one could reformulate the proofs to avoid
Grothendieck universes.  But prior to McLarty's paper, no one had.

\bigskip
\subsection{How are Grothendieck Universes used?}

\bigskip
Grothendieck needed his universes to be sets for several reasons:
\begin{enumerate}
\item to be able, using the ZFC axiom of replacement, to prove Theorem 1.10.1
of (\cite{Grothendieck1957a} 
that, in a universe, if an Abelian Category satisfies his Axiom AB5 and
admits a generator (for instance, $R$-modules over a commutative ring with unit),
then 
every object in it can be embedded in an injective object (i.e., there are ``enough injectives'');
\item to develop within a universe  the concept of derived categories; and
\item to be able within a universe to quantify over sets whose ranks were several levels
above the categories or schemes of interest in order to define and work with their cohomologies.
\end{enumerate}

To grasp better the first of these, it would be instructive to look at Grothendieck's proof
of Theorem 1.10.1 in \cite{Grothendieck1957a}, an English translation of which is available.

However, a better source for most would be the online
``The Stacks Project''  \hfil\nl (http://stacks.math.columbia.edu, Chapter 19,
Section 11).  (``The Stacks Project'' is an online, open source textbook and reference work
on algebraic stacks and the algebraic geometry needed for them.)  
Theorem 10.11.7 there proves Grothendieck's theorem on the existence of enough injectives and
also shows that the embedding into injectives is functorial, i.e., that there is a functor
${\bf M}$ such that, for each object $N$ in the category, ${\bf M}(N)$ is injective.  The results
and proofs there lay out the inductive systems that are needed and whose existence
requires the axiom of replacement.  There are also detailed discussions of the cardinality
and cofinality considerations in the proof.  The proofs themselves do not require the
existence of strongly inaccessible cardinals, but if one wants to do them within a universe
of sets, then the necessity of the axiom of replacement does require the existence of 
uncountable, strongly inaccessible cardinals in order to have a universe.

In order to see the role of replacement in a simpler setting, one can turn to the proof
of the existence of infinite injective resolutions:

\begin{thm} Let $\mathbf{A}$ be an abelian category satisfying Axiom AB5 and admitting
a generator.  Let ${\bf M}$ be the aforementioned functor that embeds each object into
an injective object.  Then for each object $N$, there is an infinite injective resolution
$$N \rightarrowtail M_0 \rightarrowtail M_1 \rightarrowtail M_2 \cdots$$
where each $M_i$ is injective and each arrow is a monomorphism.
\end{thm}

\begin{proof}
Define by induction $M_0 = {\bf M}(N)$ and $M_{n+1}={\bf M}(M_i)$.
Then $$ ZFC \proves \exists x [ x = \{M_i: i < \omega\}]$$ via the Axiom of Infinity
and the Axiom of Replacement.
\end{proof}

\bigskip
\subsection{McLarty's Strategy}
McLarty's strategy for side-stepping the aforementioned issues with Grothendieck Universes
is to develop all of Grothendieck's tools, both small and large, within a weakened set theory
together with a superstructure of simple types on top of the set theory. 
His thesis is that all
of the results in algebraic geometry  and algebraic number theory
that use any of Grothendieck's tools
can be carried through in this new foundation.

\bigskip
\subsection {Set Theories:  Zermelo-Fraenkel with Choice and MacLane Set Theory with Choice}

In order to obtain a set that is a ``universe'', i.e.,
that is closed under all of the operations of the set theory,
McLarty shifts from Zermelo-Fraenkel set theory with the Axiom of Choice to a weaker
set theory -- MacLane set theory with the Axiom of Choice, introduced by Saunders MacLane
(\cite{MacLaneMoerdijk1992}).
Zermelo set theory lies between these two set theories.
Here is a concise comparison of their axioms:

\begin{small}
\begin{center}
\begin{tabular}{r l|c|c|c|}
Axiom&Formula&Zermelo&Zermelo&McLarty's\\
&&-Fraenkel&$+$ Choice&MacLane\\
&&$+$ Choice&&$+$ Choice\\
\hline
Extensionality:&$[x=y \bicond \forall t [t\in x \bicond t \in y]]$&
$\checkmark$&$\checkmark$&$\checkmark$\\
\hline
Null set:&$x \notin \emptyset$&$\checkmark$&$\checkmark$&$\checkmark$\\
\hline
Pair:&$\forall x \forall y\exists z[x\in z \wedge y\in z \wedge \forall t[t\in z \cond t=x \vee t=y]]$
&$\checkmark$&$\checkmark$&$\checkmark$\\
\hline
Union:&$\forall x \exists z \forall t[t\in z \bicond \exists y [t\in y \wedge y \in x]]$
&$\checkmark$&$\checkmark$&$\checkmark$\\
&i.e., $z = \cup x$&&&\\
\hline
Power Set:&$\forall x \exists z \forall t[t\in z \bicond t \subseteq x]$
&$\checkmark$&$\checkmark$&$\checkmark$\\
&i.e., $z = \mathcal{P}(x)$&&&\\
\hline
Infinity& $\exists x(x=\omega)$
&$\checkmark$&$\checkmark$&$\checkmark$\\
\hline
Choice:&$\forall x(\forall y \in x(\exists z (z\in y))$ &$\checkmark$&$\checkmark$&$\checkmark$\\
&\quad $\cond \exists f(f \text{ is a function } \wedge \forall y \in x( f(y) \in y)))$&&&\\
\hline
Replacement Schema&$\forall u \forall v \forall w[\psi(u,v) \wedge \psi(u,w) \cond v = w] $
&$\checkmark$&&\\
&$\strut\quad\cond \forall z \exists y \forall v[v\in y \bicond (\exists u \in z) \psi(u,v))]$&&&\\
&where no free occurrences of $y$, $z$, and $w$ in $\psi(u,v)$&&&\\
\hline
Unbounded Separation&$\forall z \exists y\forall x[x\in y \bicond x \in z \wedge \phi(x)]$
&&$\checkmark$&\\
&where no free occurrences of $y$ in $\phi(x)$&&&\\
\hline
Bounded Separation&$\forall z \exists y\forall x[x\in y \bicond x \in z \wedge \phi(x)]$
&&&$\checkmark$\\
&where no free occurrences of $y$ in $\phi(x)$&&&\\
&and every quantifier in $\phi(x)$ is of the form&&&\\
&$\forall u \in v$ or $\exists u \in v$ and $v$ is distinct from $x$&&&\\
\hline
Foundation schema:&$\exists x \phi(x) \cond \exists x[\phi(x) \wedge (\forall y \in x)\neg \phi(y)]$
&$\checkmark$&&\footnote{Although MacLane included Foundation in his version of MacLane set theory, McLarty omits it because it is not
needed for his work on Grothendieck universes.}\\
&where $y$ is not free in $\phi(x)$&&&\\
\hline
\end{tabular}
\end{center}
\end{small}

Neither Zermelo set theory nor MacLane set theory can give the above proof for the existence
of injective resolutions.  In the next section, we will see how McLarty gets around this problem
in MacLane set theory.

\bigskip
\strut

\bigskip
\subsection {Sketch of Categories and Cohomology in MacLane Set Theory and Proofs concerning Injectives}

\bigskip
\strut

The primary problem/challenge with working in MacLane set theory is that, without replacement,
one cannot merely write down an infinitary, inductive definition and then treat the class
thereby defined as if it were a set.  In the earlier discussion, the expression
$$\{M_i : i < \omega\}$$
defines a class, but the class may not be a set, i.e., there may not be any set $a$
such that $$a = \{M_i: i < \omega\}\ .$$

In MacLane set theory, any expression that involves an index ranging over a possibly
 infinite set,
e.g., $\{X_i: i\in I\}$ where I might be an infinite set,
is suspect and must be treated with great care.
In order to show that it is a set, one must show that it can be obtained via the Axiom
of Bounded Separation.  This requires two things:
\begin{enumerate}
\item a formula $\phi(x)$ that does not contain any free occurrences of a variable $y$ 
and in which every quantifier is bound, i.e., of the form $\forall u \in v$ or
$\exists u \in v$, and
\item an ambient set $b$
\end{enumerate} such that the desired set is
$$\{x : x \in b \wedge \phi(x)\}\ .$$

So, to obtain a set of sets indexed by an index set $I$, one needs first to have
a set $X$ that already
contains all of the elements that one wants to have in the indexed sets.
Then one takes a function $s: X \rightarrow I$.  Then one defines, for each $i$ 
in $I$,
$$X_i = \{x : x \in X \wedge s(x) = i\} \ .$$
One still does not have the set $\{X_i: i \in I\}$, but one can skillfully use the function $s$.

McLarty, with great care, which we must gloss over,
shows that all of the fundamental concepts and tools of 
cohomology on small categories can be developed in MacLane set theory:
\begin{enumerate}
\item Indexed sets of small categories and of functions between them;
\item Diagrams and presheaves;
\item Natural transformations of presheaves;
\item Indexed sets of presheaves and indexed sets of natural transformation between them;
\item Indexed limits and colimits for indexed sets of presheaves on small categories;
\item The Yoneda Lemma;
\item Coverages on small categories;
\item Small sites, i.e., a small category with a coverage;
\item J-sheaves on small sites;
\item All theorems of elementary tops theory for sheaves over a small site;
\item The Functorality of presheaves and the standard six functions;
\item Comma categories;
\item \'Etale covers and fundamental groups
\item The Yoneda Lemma;
\item Coverages on small categories;
\item Small sites, i.e., a small category with a coverage;
\item J-sheaves on small sites;
\item All theorems of elementary tops theory for sheaves over a small site;
\item The Functorality of presheaves and the standard six functions;
\item Comma categories; and
\item \'Etale covers and fundamental groups.
\end{enumerate}
For the most part, the standard proofs go through after one has made fairly 
straight forward adjustments for the requirements of bounded separation.

But the situation changes drastically when one come to the existence of enough injectives
and of infinite injective resolutions.

The proof of the existence of enough injectives for Abelian Categories, with its essential
use of replacement, cannot be carried over to MacLane set theory.

Instead, one shifts attention to rings and modules and sheaves of rings and modules.

Without using replacement, one can prove that an abelian group is injective in the category of 
abelian groups if and only it is divisible, and one can prove, again without replacement, that
every abelian group can be embedded in a divisible group.

Next, using a result due to Kan -- if a functor ${\bf F:  B \longrightarrow A}$ has
a left exact left adjoint ${\bf G : A \longrightarrow B}$ with monic unit and each object
in ${\bf B}$ embeds in an injective, then so does each object in ${\bf A}$ -- one can 
bootstrap to the result that, for any ring $R$, every $R$-module embeds in an injective.

This is described in most textbooks on Homological Algebra.

But there is no easy route to infinite injective resolutions.

McLarty proceeds as follows:

\begin{quote}
\begin{proof}
Define sequences $I_i$ and $M_i$ inductively:
\begin{enumerate}
\item Set $M_0 = M$.
\item Embed $M_i$ as an additive group into a divisible group $M_i \rightarrowtail M_{di}$ .
\item Form the injective $R$-module $I_i = Hom_Z (R, M_{di} )$ with monic $M_i \rightarrowtail  I_i$ .
\item Start again, with the quotient $M_{i+1} = I_i/M_i$.
\end{enumerate}
Textbooks immediately conclude there are infinite injective resolutions, by implicit use
of (countable) replacement. MacSet [i.e., MacLane set theory] proves the same conclusion,
but only after bounding
the infinite procedure inside one ambient set for each module $M$.

The ambient will be the function set $M^{\mathbb{Z}\times R^n}$ which has an $R$-module structure induced
by $M$. Here $R^{\mathbb{N}}$ is the set of infinite sequences in $R$. Say a function
$f : \mathbb{Z} \times R^{\mathbb{N}} \rightarrow M$ is cut
off at $n \in \mathbb{N}$ if $f(m, \sigma) = 0$ for every sequence $\sigma$ which does not have
$\sigma(i) = 0$ for all 
$i \ge n$. In effect a function cut off at $n$ is an element of $M^{\mathbb{Z}\times R^N}$.
So, a function cut off at
$n + 1$ can also be regarded as a function from R to the set $M^{\mathbb{Z}\times R^n}$ 
of functions cut off at $n$.
Also, notice Step 2 is idle for $i \ge 1$ since all $I_i$ and all $M_{i+1} = I_{i+1}/I_i$ are divisible
groups. So it suffices to give an infinite injective resolution for each module $M$ with
divisible underlying group. For this case $M_i = M_{di}$ for all $i \in N$.
For any ring $R$, and $R$-module $M$ with divisible underlying group, define this induction
parallel to the one above:

\begin{enumerate}
\item[(1')] Let the subset $N_0 \subset M^{R^{\mathbb{N}}}$ contain just the additive functions
cut off at 0. In effect
these are additive functions $Z \rightarrow M$, so $N_0\cong M$.

\item[(1'')] Define equivalence relation $E_0$ as the identity on $N_0$ . The point is
$$M \cong  N 0 \cong N_0/E_0 .$$
\item[(3')] Given the subset $N_i \subset M^{R^{\mathbb{N}}}$ with every function cut off at $i$,
and equivalence
relation $E_i$ on it, define a certain subset $J_i \subset M^{R^{\mathbb{N}}}$  of functions
which are cut
off at $i + 1$. Namely, think of these as functions $R \rightarrow M^{\mathbb{Z}\times R^n}$. 
Let $Ji$ contain just
those whose values all lie in $N_i$ and which are additive when seen as functions
$R \rightarrow N_i /E_i$ . Let $Q_I$ be the pointwise equivalence relation making functions
$R \rightarrow N_i$ equivalent iff they are equal as functions $R \rightarrow N_i /E_i$ .
\item[(3'')] There is a natural monic $h : N_i \rightarrowtail J_i$ where 
for each $g \in N_i$ the value $h(g)$ is the
unique $R$-linear function $R \rightarrow N_i /E_i$ taking $1 \in R$ to $g$.
\item[(4')] Define $N_{i+1} = J_i$ with $E_{i+1}$ the smallest equivalence relation containing 
both $Q_i$ and the relation induced by the submodule $h : N_i \rightarrowtail  J_i$.
\end{enumerate}
For every $i \in \mathbb{N}$ the quotient $N_i /E_i$ is isomorphic as $R$-module to the
module $M_i$ above,
while each $J_i/Q_i$ is isomorphic to $I_i$ above, So this gives an isomorphic copy of the 
resolution by $I_i$ above. Bounded separation suffices to show this infinite resolution is one set,
since $M^{\mathbb{Z}\times R^{\mathbb{N}}}$ suffices as ambient set, and quantifier bounds
are explicit in the steps of the induction.
\end{proof}
\end{quote}

McLarty then proves 
\begin{thm} For any sheaf of rings ${\mathbf R}$  on any site $({\bf C , J})$, every
sheaf of ${\mathbf R}$-modules ${\mathbf M}$ has infinite sheaf resolutions.
\end{thm}

One can then do cohomology.

\bigskip
\subsection {Sketch of Simple Type Theory over MacLane Set Theory and Its Use and Consequences}
\strut

\bigskip
Whereas the first part of McLarty's paper could be viewed as a technical undertaking to circumvent
the problem of needing strongly inaccessible cardinals, the second part is a more radical
undertaking.  In order to meet the needs of geometers to work with and to quantify over classes
that are not sets, McLarty aims to bring together two streams of foundations that have been apart 
for a century.

In response to the discovery of Russell's antinomy, two groups emerged.  The set theorists aimed
to avoid the problem by axiomatizing which operations would lead from sets to sets without
causing any contradictions.  This led to today's Zermelo-Fraenkel set theory.

The other group, the type theorists, aimed to avoid the problem by assigning a type to each
collection so that a collection would have a higher type than each of its members.  There would 
be no set of all sets, because the collection of all sets would be of a type higher than sets.

McLarty aims to bring these two, sometimes competing, streams back together by building a simple
theory of types on top of MacLane set theory.  By doing this carefully, he is able to provide 
the Grothendieck's tools for ``large'' categories and a ``universe'' for those tools within 
a theory that is a conservative extension of MacLane set theory.

\bigskip
\subsection{McLarty's Simple Type Theory}

\bigskip
McLarty follows the path laid out by Takeuti (\cite{Takeuti1978},\cite{Takeuti1987}).
He refers to his type theory as ``MacClass''.

There is a linear hierarchy of types:
\begin{itemize}
\item There is a ground type ${\mathbf Sets}$.
\item For every type $\tau$, there is a type $[\tau]$.
\end{itemize}

So $[{\mathbf Set}]$ is the type of classes of sets, and $[[{\mathbf Set}]]$
is the type of classes of classes of sets.  For suggestive convenience,
McLarty denotes these as
$$ \mathbf{Class} = [\mathbf{Set}] \qquad \text{and} \qquad \mathbf{Collection} =
[\mathbf{Class}]\ .$$

The terms and formulas are defined by a simultaneous induction.

\begin{itemize}
\item Terms
\begin{itemize}
\item Any term of MacLane set theory is a term of type $\mathbf{Set}$.  The language
will include constant symbols $\emptyset$ and $\mathbb{N}$ and function symbols
$\cup$, $\times$, $\mathbf{P}$.  It may also include bounded set abstractions of the
form $\{x : \phi(x)\}$ where all quantifiers in $\phi$ are bounded.
\item Variables of any type are terms of that type.
\item For any formula $\Psi(\mathbf{v})$ with $\mathbf{v}$ and any free variables being
of type $\tau$ and with no quantifiers except possibly over variables of type $\mathbf{Set}$,
the formula $\{v : \Psi(v) \}$ is a term of type $[\tau]$.  The expression 
$\{v : \Psi(v) \}$ will be referred to as a \textit{set theoretic abstract}.
\end{itemize}
\item Formulas
\begin{itemize}
\item Formulas of MacLane set theory are formulas of MacClass.
\item For terms $t_1$ and  $t_2$ of type $\tau$ and and $t_3$  of type $[\tau]$,  the 
formulas $t_1 = t_2$ and $t_1 \in t_3$ are formulas of MacClass.
\item If $A$ and $B$ are formulas of MacClass, then $(\neg A)$, $(A \wedge B)$, $(A \vee B)$,
$(A \implies B)$, $\forall x A(X)$, and $\exists x A(x)$ are formulas of MacClass.
\end{itemize}
\end{itemize}

\begin{itemize}
\item Axioms:  The axioms of MacClass are the axioms of MacLane Set Theory.
\item Proofs:
\begin{itemize}
\item May use the axioms of MacLane Set Theory.
\item The formulas $t_1 \in \{v_1 : \Psi(v_1)\}$ and $\Psi(t_1)$
for any formula $\Psi(v_1)$ and term $t_1$ of the same type as $v_1$ imply each other.
\item The standard natural deduction rules for logical connectives.
\item The standard rules for $\exists$.
\end{itemize}
\end{itemize}

A formula is said to be \definedterm{set theoretic} if it only quantifies over sets, i.e.,
set variables., but it may include terms of any type.  It follows from the above axioms and rules 
of proof that for any set-theoretic formula $\Psi(v_1)$, which may have free variables other
than $v_1$, it is provable that
$$\exists a \left( v_1 \in a \bicond \Psi(v_1)\right)\ .$$

For more details, see \cite{Takeuti1987}.

Because MacClass uses only set-theoretic formulas in abstractions $\{ v_1 : \Psi(v_1)\}$,
Gentzen cut elimination shows the MacClass is conservative over MacLane set theory
(\cite{Takeuti1987}, page 176).
MacClass can quantify over classes in proofs, but it cannot quantify over classes in definitions 
of sets and classes.

To facilitate reading formulas, McLarty adopts the following conventions:
\begin{itemize}
\item Variables of type $\mathbf{Set}$ will be denoted by math italics such as $x$ and $A$.
\item Variables of type $\mathbf{Class} = [\mathbf{Set}]$ will be denoted by caligraphic
letters $\mathcal{A}$, $\mathcal{B}$, etc.
\item Variables of type $\mathbf{Collection} = [\mathbf{Class}] = [[\mathbf{Set}]]$ will be
denoted by fraktur letters $\mathfrak{A}$, $\mathfrak{B}$, etc.
\end{itemize}

Relation symbols may be subscripted to indicate typing.  For instance
\begin{enumerate}
\item $A \in_0 B$ indicates that $A$ is a set that is a member of the set $B$.
\item $A \in_1 \mathcal{A}$ indicates that the set $A$ is a member for the class $\mathcal{A}$.
\item $\mathcal{A} \in_2 \mathfrak{A}$ indicates that the class $\mathcal{A}$ is a member of the collection $\mathfrak{A}$
\item $\mathcal{A} \subseteq_1 \mathcal{B} \bicond \forall x(x\in_1 \mathcal{A} \cond x \in_1 \mathcal{B})$
defines inclusion for classes.
\item $\mathfrak{A} \subseteq_1 \mathfrak{B} \bicond \forall x(x\in_1 \mathfrak{A} \cond x \in_1 \mathfrak{B})$
defines inclusion for collections.
\item $A \subseteq_{01} \mathcal{B} \bicond \forall x(x\in_0 A \cond x \in_1 \mathcal{B})$ defines when a set $A$
is a subclass of of the class $\mathcal{B}$.
\end{enumerate}

Note that every set $A$ defines a class $\mathcal{A}$ and a collection $\mathfrak{A}$ with the same elements:
\begin{align*}
\forall x (x\in_0 A &\bicond x \in_1 \mathcal{A})\\
\forall x (x\in_0 A &\bicond x \in_2 \mathfrak{A})\\
\end{align*}

In these situations, one may say informally that $\mathcal{A}$ and $\mathfrak{A}$ are sets.

\bigskip
\subsection {Categories and the Universe}

\bigskip
After going through a bit more detail, one can show that there is a class $\mathcal{Cat}$
consisting of all small categories and there is a collection $\mathfrak{Cat}$
consisting of all categories that are classes.

Finally, one takes as the \definedterm{Universe $\mathcal{U}$} the class of all sets.

This provides the universe that Grothendieck wanted.

\bigskip
\subsection{Next Steps}
McLarty  goes on to develop all of the large concepts and tools within MacClass,
where they have the ``meaning'' that was intended by Grothendieck and other geometers.

\bigskip
\section{Conclusion}
McLarty has developed a foundation for algebraic geometry consisting of MacLane Set Theory
and his MacLane Class Theory that provides
\begin{enumerate}
\item a set theory that is adequate for dealing with small categories without invoking any
large cardinal assumption and
\item a class theory that is adequate for the large categories and toposes used by geometers.
\end{enumerate}
Moreover, the class theory is a conservative extension of the set theory.

\bigskip
\end{Large}

\bigskip
\bigskip

\bibliography{McLarty}{}
\bibliographystyle{plain}
\end{document}